\documentclass[a4paper,reqno,12pt]{amsart}
	\usepackage[latin1]{inputenc}
    \usepackage[T1]{fontenc}
    \usepackage[english]{babel}
    \usepackage{appendix}
    \usepackage[english]{minitoc}
    \usepackage[nice]{nicefrac}
\usepackage[top=2.5cm, bottom=2.5cm, left=2.5cm, right=2.5cm]{geometry}
	\usepackage{amsmath}
    \usepackage{amssymb}
    \usepackage{amsbsy} 
    \usepackage{latexsym, amsfonts}
    \usepackage{graphics}
    \usepackage{ulem}
    \usepackage{hhline}
    \usepackage{dsfont}
    \usepackage{mathrsfs}
    \usepackage{color}
    \usepackage{fancyhdr}
    \usepackage{rotating}
    \usepackage{fancybox}
    \usepackage{colortbl}
    \usepackage{pifont}
    \usepackage{setspace}
    \usepackage{enumerate}
    \usepackage{multicol}
    \usepackage{varioref}
    \usepackage{lmodern}
    \usepackage{textcomp}
    \usepackage{euscript}
    \usepackage[pdftex]{hyperref}
			\newtheorem{theorem}{Theorem}[section]

			\newtheorem{lemma}[theorem]{Lemma}

			\newtheorem{remark}[theorem]{Remark}

    \newcommand{\R}{\mathbb R}    	 

\newcommand{\minmn}{m\wedge n}  
 
\begin{document}

\title{Generalized Zernike polynomials: Integral representation and Cauchy transform}
\thanks{A. Ghanmi and A. Intissar are partially supported by the Hassan II Academy of Sciences and Technology.}
\author{A. El Hamyani}  \email{amalelhamyani@gmail.com}
\author{A. Ghanmi}     \email{ag@fsr.ac.ma}
\author{A. Intissar}   \email{intissar@fsr.ac.ma};\email{ahmedintissar@gmail.com} 
\maketitle
\address{P.D.E. and Spectral Geometry,
          Laboratory of Analysis and Applications - URAC/03,
          Department of Mathematics, P.O. Box 1014,  Faculty of Sciences,
          Mohammed V University, Rabat, Morocco}

\begin{abstract}
The aim of this paper is two fold. We derive an integral representation for the generalized 2D Zernike polynomials which are of
independent interest and give the explicit expression of the action of the Cauchy transform on them.
\end{abstract}


\section{Introduction} 

The disk polynomials in the two conjugate complex variables $z$ and $\overline{z}$ in the unit disk $D$ of the complex plane, up to a multiplicative constant, are defined by the Rodrigues' formula
  \begin{align}\label{ZernikePol}
  \mathcal{Z}_{m,n}^\gamma (z,\bar z)
   = (-1)^{m+n}  (1- |z|^2)^{-\gamma} \frac{\partial^{m+n}}{\partial z^m \partial \bar z^n}\left( (1- |z|^2)^{\gamma+m+n} \right).
  \end{align}
This definition agrees with W\"unsche \cite{Wunsch05} ($ \displaystyle P_{m,n}^{\gamma}(z, \bar z)$), Dunkl \cite{Dunkl83,Dunkl84} ($ \displaystyle R_{m,n}^{(\gamma)}(z)$)  and up to standardization with Koornwinder \cite{Koornwinder75,Koornwinder78}, being indeed
$$\mathcal{Z}_{m,n}^\gamma (z,\bar z)  = (\gamma+1)_{m+n} \overline{P_{m,n}^{\gamma}(z, \bar z)}= (\gamma+1)_{m+n} \overline{R_{m,n}^{(\gamma)}(z)} .$$

They form a complete orthogonal system (basis) over the Hilbert space $L^{2,\gamma}(D):=L^2(D;(1-|z|^2)^\gamma dxdy)$, where $\gamma>-1$, and are often referred to as generalized (or 2D) Zernike polynomials. Indeed, for $\gamma=0$ and $m\leq n$, the $\mathcal{Z}_{m,n}^\gamma$ turn out to be related to the real Zernike polynomials  $R^\nu_k(x)$ introduced by Zernike in his framework on optical problems involving telescopes and microscopes, to wit
$$\mathcal{Z}_{m,n}^0 (z,\bar z) =   (m+n)! e^{i[(n-m)\arg z]}  R^{n-m}_{m+n}(\sqrt{z\bar z}),$$
which play an important role in expressing the wavefront data in optical tests 
and in the study of diffraction problems (see \cite{ZernikeBrinkman35}).

An accurate analysis of the basic analytic properties of $ \displaystyle \mathcal{Z}_{m,n}^\gamma (z,\bar z)$, like recurrence relations with respect to the indices $m$ and $n$, differential equations they might obey, some generating functions and so on, have been developed in various papers from different point of views. See \cite{Koornwinder75,Koornwinder78} by Koornwinder for a very nice account on these polynomials and \cite{Wunsch05} for an elegant reintroduction of them. Recently, a Wiener type theorem and a Paley type theorem for the disk polynomial expansions have been obtained by Kanjin in \cite{Kanjin2013}.
More recently, new operational formulae of Burchnall type and generating functions are obtained in \cite{Aharmim2015}.

The purpose of the present paper is to first derive a new integral representation for the generalized 2D Zernike polynomials, involving a modified Blaschke function restricted to the unit circle. Second, we  give the explicit expression of the Cauchy transform
\begin{align*}
\left[C_{\gamma} (f) \right](z) &= \dfrac{1}{\pi} \int_{D}\dfrac{f(w)}{w-z}  \left(1-|{z}|^{2}\right)^{\gamma}dxdy
\end{align*}
 on the generalized Zernike polynomials. More precisely, we show the following
  \begin{align}
\left[ C_{\gamma} (\mathcal{Z}_{m,n}^\gamma) \right] (z) &=
   \left(1-|z|^2\right)^{\gamma+1} \mathcal{Z}_{m,n-1}^{\gamma+1}(z,\bar{z}) ; \quad m\geq n.
\end{align}

The remaining sections are organized as follows. In Section 2, we review the Schr\"odinger's algebraic method that we use to generate the disk polynomials $\mathcal{Z}_{m,n}^\gamma (z,\bar z)$. In Section 3, we review some basic analytic properties and establish an integral representation for this class of polynomials. Section 4 is devoted to the action of the Cauchy transform $C_{\gamma}$ on the $\mathcal{Z}_{m,n}^\gamma$.

\section{Generation of the generalized Zernike polynomials}
We consider the magnetic Laplacian (called also twisted Laplacian)
\begin{eqnarray*}
 \mathfrak{L}_\nu  = -(1 - |z|^2)^2\frac {\partial^2} {\partial z\partial \bar z}
 - \nu (1 - |z|^2)\left(z\frac {\partial}{\partial z}- {\bar z}\frac{\partial}{\partial{\bar z}}\right) + \nu^2 |z|^2; \quad |z|<1,
\end{eqnarray*}
acting on the $L^2$-Hilbert space $L^2(D ;(1-|z|^2)^{-2}dxdy)$, where $D$ is the unit disk in the complex plane.
It can be seen as a magnetic Schr\"odinger operator with respect to the hyperbolic geometry of the disc $D$ provided by the hyperbolic metric $ds^2 := (1 - |z|^2)^{-2}dz\otimes d\bar z$ and associated to the vector potential given by
$$i\nu \theta(z) = i\nu (\partial -\overline{\partial}) \log (1-|z|^2)
= \dfrac{-i\nu (\bar z dz - z d\bar z) }{1-|z|^2}  = \dfrac{-2\nu (y dx - x dy) }{1-x^2-y^2}  .$$
 Indeed, we have
\begin{eqnarray*}
 \mathfrak{L}_\nu = (d+i\nu ext(\theta))^{*}(d+i\nu ext(\theta)),
\end{eqnarray*}
where $d$ and $ext(\theta)$ are respectively the differential operator and the exterior multiplication by the differential $1$-form $\theta$. The adjoint operation is taken with respect to the Hermitian scalar product on compactly supported differential forms
\begin{align*} \label{sp}
(\alpha,\beta):=\int_{D}\alpha\wedge\star \beta,
\end{align*}
 where $\star $ is the Hodge star operator canonically associated with the hyperbolic metric on the unit disc.
Thus, $ \mathfrak{L}_\nu$ is an elliptic self-adjoint second order differential operator and their spectral properties are well known in the literature \cite{Zhang92,GhIn2005}. Its discrete $L^2$-spectrum is nontrivial if and only if $\nu > 1/2$. It is given through the eigenvalues
$$\lambda_{\nu,m}:=\nu(2m+1) - m(m+1)$$
for varying positive integer $m$ such that $0\leq m< \nu -1/2$.

By considering the first order differential operator $\nabla_\alpha$ and its formal adjoint
$\nabla_\alpha^{*}$ given by
\begin{align*}
\nabla_\alpha = -(1 - |z|^2)\frac{\partial}{\partial z} + \alpha \bar z \quad \mbox{and} \quad
\nabla_\alpha^{*}  = (1 - |z|^2)\frac{\partial}{\partial \bar z} + (\alpha +1) z,
\end{align*}
we can factorize the twisted Laplacian $\mathfrak{L}_\nu$ as (\cite{FV2001,Gh08OM})
 \begin{align*}
\mathfrak{L}_\nu  = \nabla_{\nu}^{*}  \nabla_{\nu}  - \nu =     \nabla_{\nu-1}  \nabla_{\nu-1}^{*}  + \nu.
 \end{align*}
Therefore, we have the following algebraic relationship 
  \begin{align*}
\mathfrak{L}_\nu \nabla_{\nu-1} = \left(\nabla_{\nu-1}  \nabla_{\nu-1}^{*} + \nu \right) \nabla_{\nu-1}
                               = \nabla_{\nu-1} \left(  \mathfrak{L}_{\nu -1} +(2\nu-1) \right),
 \end{align*}
so that the first order differential operator $\nabla_{\nu-1}$ allows one to generate eignefunctions of $\mathfrak{L}_{\nu}$ from those of $\mathfrak{L}_{\nu -1}$. More generally, if $\varphi_{0}$ is a nonzero $L^2$-eigenfunction associated to the lowest eigenvalue of $\mathfrak{L}_{\nu - m}$, then
$$\nabla^{\nu}_{m}\varphi_{0}= \nabla_{\nu-1} \circ \nabla_{\nu-2}\circ \cdots \circ \nabla_{\nu - m}\varphi_{0}$$
 is an $L^2$-eigenfunction of $\mathfrak{L}_\nu $.
 Moreover, for fixed $\nu>\frac{1}{2}$ and varying $n=0,1,2, \cdots $, the functions
 \begin{align} 
  \psi_{m,n}^{\nu}(z,\bar z):= \nabla^{\nu}_{m} \left(z^n(1 - |z|^2)^{\nu-m} \right)
\end{align}
  are $L^2$-eigenfunctions of $\mathfrak{L}_\nu$. Furthermore, they constitute an orthogonal basis of the $L^2$-eigenspace
$$ A^{2,\nu}_m(D):=\big\{\varphi \in L^2(D ; (1 - |z|^2)^{-2} dxdy);  ~~ \mathfrak{L}_\nu  \varphi = \left(\nu(2m+1) - m(m+1)\right) \varphi  \big\}.$$

The suggested class of two variable polynomials are
\begin{align*}
\mathcal{Z}_{m,n}^\gamma (z,\bar z) : &= (\gamma+m+1)_n(1 - |z|^2)^{-\nu  +m}\nabla^{\nu}_{m} \left( z^n (1 - |z|^2)^{\nu-m}\right)\\
&=(-1)^{m+n}\left(1-\mid z\mid^{2}\right)^{-\gamma}\dfrac{\partial^{m+n}}{\partial z^{m}\partial \overline{z}^{n}}\left(\left(1-|z|^{2}\right)^{\gamma+m+n}\right),
\end{align*}
where $\gamma=2(\nu-m)-1$.
Their explicit expression are given by
 \begin{align*}
   \mathcal{Z}_{m,n}^{\gamma}(z,\bar{z})
 =  m!n!  (\gamma+1)_{m+n} \sum_{j=0}^{\minmn }
    \frac{(-1)^{j} (1 -  |z|^2)^{j}} {(\gamma+1)_j j!} \frac{\bar z^{m-j}}{(m-j)! } \frac{z^{n-j} }{(n-j)!}\,.
  \end{align*}

\section{A new integral representation for $\mathcal{Z}_{m,n}^\gamma$}

The polynomials $\mathcal{Z}_{m,n}^\gamma (z,\bar z)$ can be written in terms of the Gauss-hypergeometric function ${_2F_1}$ starting from explicit expression. Indeed, we have  (\cite[p. 137]{Wunsch05})
    \begin{align}\label{ZernikeGauss1}
   \mathcal{Z}_{m,n}^{\gamma}(z,\bar{z})  =  (\gamma+1)_{m+n} \overline{z}^m z^n
   {_2F_1}\left( \begin{array}{c} -m , -n \\ \gamma+1 \end{array}\bigg | 1-\frac 1{|z|^2} \right)=(\gamma+1)_{m+n} \overline{P_{m,n}^{\gamma}(z, \bar z)}
  \end{align}
   from which we can recover the well-known ${_2F_1}$ formula (see \cite[p. 692]{Dunkl83} or \cite[p. 535]{Dunkl84}),
     \begin{align}\label{ZernikeGauss2}
   \mathcal{Z}_{m,n}^{\gamma}(z,\bar{z})  =  \frac{\left((\gamma+1)_{m+n}\right)^2} {(\gamma+1)_{m}(\gamma+1)_{n}}\overline{z}^m z^n
   {_2F_1}\left( \begin{array}{c} -m , -n \\ -\gamma-m-n \end{array}\bigg | \frac 1{|z|^2} \right) = (\gamma+1)_{m+n}  \overline{R_{m,n}^{(\gamma)}(z)}.
  \end{align}
Their expression involves also the real Jacobi polynomials defined by $$P_{n}^{(\alpha,\beta)}(x)=\left(1-x\right)^{-\alpha}\left(1-x\right)^{-\beta}\dfrac{(-1)^{n}}{2^{n}n!}
\dfrac{d^{n}}{dx^{n}}\bigg\{\left(1-x\right)^{n+\alpha}\left(1-x\right)^{n+\beta}\bigg\}.$$
More precisely, we have (\cite{Gh08OM,Kanjin2013})
\begin{align}\label{ZernikeJacobi}
\mathcal{Z}_{m,n}^\gamma (z,\bar z) = (-1)^{m}  (m\wedge n )!(\gamma+m+1)_n |z|^{|m-n|}e^{i[(n-m)\arg z]}
\mathrm{P}^{(|m-n|,\gamma)}_{m\wedge n }(1-2 |z|^2).
\end{align}
The both expressions, \eqref{ZernikeGauss1} (or \eqref{ZernikeGauss2}) and \eqref{ZernikeJacobi}, can be used to derive integral representations for $\mathcal{Z}_{m,n}^\gamma (z,\bar z)$. However, using elementary facts we give a new integral representation which is of independent interest. Namely, we have

\begin{theorem}
The following integral representation
\begin{align} \label{IntRepZ}
\mathcal{Z}_{m,n}^{\gamma}(z,\overline{z})= \dfrac{-(\gamma+m+1)_{n} m! }{2\pi i } \left(1- |z|^{2}\right)^{-\gamma}
\oint_{\mid t\mid=1} t^{n} \dfrac{\left(1- {t\overline{z}}\right)^{\gamma+m}}{(z-t)^{m+1}}dt
\end{align}
holds for the generalized Zernike polynomials $\mathcal{Z}_{m,n}^{\gamma}(z,\overline{z})$.
\end{theorem}

\begin{proof} For the proof, we proceed as in \cite{KazantsevBukhgeim04}.
Start from
\begin{align*}
\mathcal{Z}_{m,n}^{\gamma}(z,\overline{z})= (-1)^{m}(\gamma+m+1)_{n} \left(1-|{z}|^{2}\right)^{-\gamma}\dfrac{\partial^{m}}{\partial z^{m}}\left(z^{n}\left(1-|{z}|^{2}\right)^{\gamma+m}\right)
\end{align*}
and use the ordinary binomial expansion with the factorial function  $(1- \xi)^{-a}=\sum\limits_{j=0}^{+\infty}\dfrac{(a)_{j}}{j!} {\xi^{j}}$ 
to get
\begin{align*}
\mathcal{Z}_{m,n}^{\gamma}(z,\overline{z})= {(-1)^{m}(\gamma+m+1)_{n}} \left(1-|{z}|^{2}\right)^{-\gamma}
\sum_{j=0}^{+\infty}\dfrac{(-\gamma-m)_{j}}{j!}\dfrac{\partial^{m}}{\partial z^{m}}(z^{j+n}) {\overline{z}^{j}} .
\end{align*}
Now, since  $$\dfrac{\partial^{m}}{\partial z^{m}}(z^{j+n})=\dfrac{m!}{2\pi i}\oint_{\mid t\mid=1}\dfrac{t^{j+n}}{(t-z)^{m+1}}dt,$$
it follows
\begin{align*}
\mathcal{Z}_{m,n}^{\gamma}(z,\overline{z}) =\frac{(-1)^{m}(\gamma+m+1)_{n}m!}{2\pi i }\left(1-\mid {z} \mid^{2}\right)^{-\gamma}
\oint_{\mid t\mid=1}\dfrac{t^{n}}{(t-z)^{m+1}}\left(\sum_{j=0}^{+\infty}\dfrac{(-\gamma-m)_{j}}{j!} {t^{j}\overline{z}^{j}}\right) dt,
\end{align*}
which gives rise to \eqref{IntRepZ}.
\end{proof}

\begin{remark}
As immediate consequence, we have
\begin{align}
\mathcal{Z}_{m,n}^{\gamma}(0,0) &= (-1)^m(\gamma+m+1)_{m} m! \delta_{m,n}.
\end{align}
This can be used to recover the well known fact that $H_{m,n} (0,0) = (-1)^m m! \delta_{m,n}$ (see for example \cite[Eq (3.13) Theorem 3.3]{IsmailArxiv02-2015}) for the complex Hermite polynomials defined by
  \begin{align*}
  H_{m,n} (z,\bar z)
   := (-1)^{m+n}  e^{|z|^2} \frac{\partial^{m+n}}{\partial z^m \partial \bar z^n}\left( e^{-|z|^2}\right).
  \end{align*}
  Indeed, starting from the fact that
$$  \lim\limits_{\rho \to +\infty} \dfrac {\mathcal{Z}_{m,n}^{\rho^2}(z/\rho,\bar z/\rho)}{\rho^{m+n}}
= H_{m,n} (z,\bar z)$$
and using the Binet formula, we get
\begin{align*}
H_{m,n} (0,0) = \lim\limits_{\rho \to +\infty} \frac {\mathcal{Z}_{m,n}^{\rho^2}(0,0)}{\rho^{m+n}}
               = \lim\limits_{\rho \to +\infty} (-1)^m m! \delta_{m,n} \dfrac{\Gamma(\rho^2+2m+1)}{\rho^{2m} \Gamma(\rho^2+m+1)}
              = (-1)^m m! \delta_{m,n}.
\end{align*}
\end{remark}

\section{Cauchy transform and the generalized Zernike polynomials}

For $\gamma>-1$, we set
$$d\mu_{\gamma}(z)=\left(1-|{z}|^{2}\right)^{\gamma}dxdy,$$
where $z=x+iy$; $x,y\in \R$. We define the Cauchy transform $C_{\gamma}$ of a given $f\in L^{2}\left(D,d\mu_{\gamma}\right)$ by
\begin{align*}
[C_{\gamma}(f)](z)&=\dfrac{1}{\pi} \int_{D }\dfrac{f(w)}{w-z}d\mu_{\gamma}(w).
\end{align*}

The following result gives the explicit expression of the Cauchy transform $C_{\gamma}$ on the generalized Zernike polynomials. Namely, we assert

\begin{theorem} \label{Thm:CauchyZ} For every nonnegative integers $m\geq n$ and real $\gamma >-1$, we have
\begin{align}
\left[C_{\gamma}(\mathcal{Z}_{m,n}^{\gamma} )\right](z) = \left(1-|{z}|^{2}\right)^{\gamma+1}
\mathcal{Z}_{m,n-1}^{\gamma+1}(z,\overline{z}) .
\end{align}
\end{theorem}

The proof of this theorem lies on the following key lemma.

\begin{lemma} \label{Lem:CauchyMonomials}
For every nonnegative integers $m,n$ and $j$ such that $m\geq n\geq j$, we have
\begin{align}\label{CauchyMonomials}
C_{\gamma}\left(\overline{z}^{m-j}z^{n-j}\left(1- |{z}|^{2}\right)^{j}\right)
&=  \dfrac{-z^{n} \overline{z}^{m+1}}{(m-j+1)}    \left(1-|z|^2\right)^{\gamma+1}
\\& \quad \times \left(\dfrac{1-|z|^2}{|z|^{2}}\right)^{j}
  {_2F_1}\left( \begin{array}{c} 1, \gamma+m +2\\ m-j+2\end{array}\bigg ||z|^2\right). \nonumber
\end{align}
\end{lemma}

\begin{proof}
Note first that for every nonnegative integers $m,n$ such that $ m \geq n$ and fixed complex number $z\in D$
 and positive real number $r$ such that $0<r<1$, we have
\begin{align}\label{integraltheta}
\int_0^{2\pi} \dfrac{e^{i(n-m)\theta }}{re^{i\theta}-z} d\theta =
\left\{ \begin{array}{ll}
 -2\pi \dfrac{r^{m-n}}{z^{m-n+1}} &  \quad \mbox{if } 0<r<|z|\\
  0                               &  \quad \mbox{if } 0\leq |z|<r<1
\end{array}
\right.
.
\end{align}
This can be handled easily by expanding $\dfrac{1}{re^{i\theta}-z}$ as power series and next making use of the well-known fact that
$\int_0^{2\pi} e^{i(k-k')\theta } d\theta = 2\pi \delta_{k,k'}$ for every nonnegative integers $k,k'$.

To prove \eqref{CauchyMonomials}, we start from the definition of the Cauchy transform and make use of the polar coordinates $w=re^{i\theta}$ to get
\begin{align*}
C_{\gamma}\left(\overline{z}^{m-j}z^{n-j}\left(1- |{z}|^{2}\right)^{j}\right)
= \dfrac{1}{\pi} \int_0^1 r^{m+n-2j+1} (1-r^2)^{\gamma+j} \left(\int_0^{2\pi} \dfrac{e^{i(n-m)\theta }}{re^{i\theta}-z} d\theta\right)dr.
\end{align*}
Now, according to \eqref{integraltheta}, the previous integral reduces further to
\begin{align*}
C_{\gamma}\left(\overline{z}^{m-j}z^{n-j}\left(1-|{z}|^{2}\right)^{j}\right)
= \dfrac{-2}{z^{m-n+1}} \int_0^{|z|} r^{2(m-j)+1} (1-r^2)^{\gamma+j} dr.
\end{align*}
Whence, the change $t=r^2$ yields  
\begin{align}\label{Cauchy1}
C_{\gamma}\left(\overline{z}^{m-j}z^{n-j}\left(1-|{z}|^{2}\right)^{j}\right)
= \dfrac{-1}{z^{m-n+1}} \int_0^{|z|^2} t^{m-j} (1-t)^{\gamma+j} dt.
\end{align}
In the last equality, we recognize the well-known integral 
$$\int_{0}^{x}u^{a-1}(1-u)^{b-1}du=\dfrac{1}{a} x^{a} {_2F_1}\left( \begin{array}{c} a,1-b\\ a+1\end{array}\bigg |x\right),$$
that we can rewrite as
\begin{align}\label{Cauchy2}
\int_{0}^{x}u^{a-1}(1-u)^{b-1}du= \dfrac{1}{a} x^{a} \left(1-x\right)^{b} {_2F_1}\left( \begin{array}{c} 1,a+b\\ a+1\end{array}\bigg |x\right),
\end{align}
thanks to the Euler's transformation \cite[Theorem 21, p. 60]{Rain} 
$${_2F_1}\left( \begin{array}{c} a,b\\ c\end{array}\bigg |x\right)=\left(1-x\right)^{c-a-b}{_2F_1}\left( \begin{array}{c} c-a,c-b\\ c\end{array}\bigg |x\right).$$
From \eqref{Cauchy1} and \eqref{Cauchy2}, with $x=|z|^2$, $a=m-j+1$ and $b=\gamma+j+1$, we get
\begin{align*}
C_{\gamma}\left(\overline{z}^{m-j}z^{n-j}\left(1-|{z}|^{2}\right)^{j}\right)
= \dfrac{-z^{n-j} \overline{z}^{m-j+1} }{m-j+1}  \left(1-|z|^2\right)^{\gamma+j+1}
  {_2F_1}\left( \begin{array}{c} 1, \gamma+m +2\\ m+2-j\end{array}\bigg ||z|^2\right).
\end{align*}
This is exactly \eqref{CauchyMonomials}.
\end{proof}

Accordingly, we can prove Theorem \ref{Thm:CauchyZ}.

\begin{proof}[Proof of Theorem \ref{Thm:CauchyZ}]
Let $m\geq n$ and write the explicit expression of $\mathcal{Z}_{m,n}^{\gamma}(z,\overline{z})$ in the $z$ variable, to wit
$$
\mathcal{Z}_{m,n}^{\gamma}(z,\overline{z})=m!n!(\gamma+1)_{m+n}\sum_{j=0}^{n}
\dfrac{(-1)^{j} }{(\gamma+1)_{j}} \dfrac{\overline{z}^{m-j}}{(m-j)!}  \dfrac{z^{n-j}}{(n-j)!} \dfrac{\left(1-|{z}|^{2}\right)^{j}}{j!}.
$$
By linearity of the Cauchy transform, we need to compute  $C_{\gamma}\left(\overline{z}^{m-j}z^{n-j}\left(1-|{z}|^{2}\right)^{j}\right)$ which is given through Lemma \ref{Lem:CauchyMonomials} (see \eqref{CauchyMonomials}). Thus, we get
\begin{align*}
\left[C_{\gamma}(\mathcal{Z}_{m,n}^{\gamma}) \right](z)
&=  m!n! (\gamma+1)_{m+n}\sum_{j=0}^{n}\dfrac{(-1)^{j} }{ (m-i)!(n-j)! (\gamma+1)_{j} j!}
 C_{\gamma}\left({\overline{z}^{m-j}z^{n-j}} \left(1-|{z}|^{2}\right)^{j}\right)
\\&= - (\gamma+1)_{m+n} z^{n} \overline{z}^{m+1} \left(1-|z|^2\right)^{\gamma+1}
    \\& \quad \times  n! \sum_{j=0}^{n}    \dfrac{(-1)^{j}  m! }{ (m+1-j)! (\gamma+1)_{j}}
    \dfrac{\left(\dfrac{1-|z|^2}{|z|^{2}}\right)^{j}}{j!(n-j)!}
   {_2F_1}\left( \begin{array}{c}  1, \gamma+m +2\\ m+2-j\end{array}\bigg ||z|^2\right)
 \\&=  - (\gamma+1)_{m+n} \dfrac{z^{n} \overline{z}^{m+1}}{m+1} \left(1-|z|^2\right)^{\gamma+1}
    \\& \quad \times  n! \sum_{j=0}^{n}    \dfrac{(-(m+1))_j }{(\gamma+1)_{j} }
    \dfrac{\left(\dfrac{1-|z|^2}{|z|^{2}}\right)^{j}}{j!(n-j)!}
   {_2F_1}\left( \begin{array}{c}  1, \gamma+m +2\\ m+2-j\end{array}\bigg ||z|^2\right).
\end{align*}
The last equality follows making use of the fact that $(-1)^{j}(m+1)!= (-(m+1))_j (m+1-j)!$. Next, by means of \cite[p.415]{Brichkov2008}
$$
n! \sum_{j=0}^{n}    \dfrac{(1-c)_{j}}{(b-c+1)_{j}} \dfrac{\left(\dfrac{1-x}{x}\right)^{j}}{j!(n-j)!}
{_2F_1}\left( \begin{array}{c} a,b \\ c-j\end{array}\bigg |x\right)
=\dfrac{(1-c)_{n}}{(b-c+1)_{n} } x^{-n}{_2F_1}\left( \begin{array}{c} a-n,b \\ c-n\end{array}\bigg |x\right),
$$
with $a= 1$, $b= \gamma+m+2$, $c=m+2$ and $x= |z|^2$, we obtain
\begin{align*}
\left[C_{\gamma}(\mathcal{Z}_{m,n}^{\gamma}) \right](z)
&=  \dfrac{- (\gamma+1)_{m+n} (-m-1)_{n}}{(\gamma +1)_{n}  (m+1)}
   \overline{z}^{m+1-n} \left(1-|z|^2\right)^{\gamma+1}
    {_2F_1}\left( \begin{array}{c} 1-n,\gamma+m+2 \\ m-n+2\end{array}\bigg ||z|^2\right).
\end{align*}
Now, by applying \cite[Eq. (15.8.6)]{NIST}  
$$ {_2F_1}\left( \begin{array}{c} -k,b \\ c\end{array}\bigg |x\right)
=\dfrac{(b)_{k}}{(c)_{k}} (-x)^{k}{_2F_1}\left( \begin{array}{c} -k,1-c-k \\ 1-b-k\end{array}\bigg |\frac{1}{x}\right),$$
with $k=n-1$, $b= \gamma+m+2$, $c=m-n+2$ and $x=|z|^2$, it follows
\begin{align*}
\left[C_{\gamma}(\mathcal{Z}_{m,n}^{\gamma}) \right](z)
&=  \dfrac{(-1)^{n} (\gamma+1)_{m+n} (\gamma+m+2)_{n-1} (-m-1)_{n}} {(\gamma +1)_{n} (m-n+2)_{n-1}(m+1)}
   \\& \quad \times  z^{n-1} \overline{z}^{m} \left(1-|z|^2\right)^{\gamma+1}
    {_2F_1}\left( \begin{array}{c} -(n-1), - m \\ -(\gamma+1)-m-(n-1) \end{array}\bigg |\frac{1}{|z|^2}\right).
\end{align*}
According to \eqref{ZernikeGauss2}, we can write
\begin{align*}
\left[C_{\gamma}(\mathcal{Z}_{m,n}^{\gamma}) \right](z)
&=  \dfrac{(-1)^{n} (\gamma+1)_{m+n} (\gamma+m+2)_{n-1} (-m-1)_{n}} {(\gamma +1)_{n} (m-n+2)_{n-1}(m+1)} \left(1-|z|^2\right)^{\gamma+1}
   \\& \frac {(\gamma+2)_{m}(\gamma+2)_{n-1}}{\left((\gamma+2)_{m+n-1}\right)^2}\mathcal{Z}_{m,n-1}^{\gamma+1}(z,\bar{z}),
\end{align*}
which reduces further to
\begin{align*}
\left[C_{\gamma}(\mathcal{Z}_{m,n}^{\gamma}) \right](z)
&=   \left(1-|z|^2\right)^{\gamma+1} \mathcal{Z}_{m,n-1}^{\gamma+1}(z,\bar{z})
\end{align*}
thanks to the followings facts
$(\gamma+m+2)_{n-1} = (\gamma+1)_{m+n} / (\gamma+1)_{m+1} $, $(-1)^{n} (-m-1)_{n}= (m+1) (m+2-n)_{n-1}$ and $(a)_{1+k}= a (a+1)_{k}$.
\end{proof}

\begin{remark}
The case $n\geq m$ follows easily from the previous one. Indeed, by definition of the Cauchy transform $C_{\gamma}$ and the fact $\overline{\mathcal{Z}_{j,k}^{\gamma}(w,\overline{w})} = \mathcal{Z}_{k,j}^{\gamma}(w,\overline{w})$, we get
\begin{align*}
\left[C_{\gamma}(\mathcal{Z}_{m,n}^{\gamma}) \right](z) =\overline{\left[C_{\gamma}(\overline{\mathcal{Z}_{m,n}^{\gamma}}) \right](\overline{z})}
=  \overline{\left[C_{\gamma}(\mathcal{Z}_{n,m}^{\gamma}) \right](\overline{z})}.
 \end{align*}
Therefore, from Theorem \ref{Thm:CauchyZ}, we obtain
\begin{align*}
\left[C_{\gamma}(\mathcal{Z}_{m,n}^{\gamma}) \right](z)
=  \overline{ \left(1-|z|^2\right)^{\gamma+1} \mathcal{Z}_{n,m-1}^{\gamma+1} (\overline{z},z)}
= \left(1-|z|^2\right)^{\gamma+1} \mathcal{Z}_{n,m-1}^{\gamma+1}(z,\bar{z}),
 \end{align*}
 since $\overline{\mathcal{Z}_{j,k}^{\gamma} (z,\overline{z})} = \mathcal{Z}_{j,k}^{\gamma}(\overline{z},z)$.
\end{remark}

\end{document}